\documentclass[12pt]{amsart}
\usepackage{amsmath,amsthm,amssymb}

\def\pmatrix{\left(\begin{matrix}}
\def\endpmatrix{\end{matrix}\right)}

\def\G{{\mathbb G}}

\def\Z{{\mathbb Z}}
\def\C{{\mathbb C}}

\def\I{{\mathcal I}}

\def\de{\delta}

\def\p{\partial}

\def\t{\theta}

\def\T{\Theta}
\def\e{\varepsilon}

\def\a{\alpha}
\def\e{\epsilon}
\def\b{\beta}
\def\d{\delta}
\def\X{{\mathcal X}}

\def\A{{\mathcal A}}

\def\H{{\mathcal H}}
\def\G{{\mathcal G}}

\def\S{{\mathcal S}}

\def\X{{\mathcal X}}

\def\tch#1#2{{\left[\begin{matrix}{#1}\\ {#2}\end{matrix}\right]}}
\def\tt#1#2{{\t\tch{#1}{#2}}}
\def\Sp{\operatorname{Sp}(g,\Z)}
\def\Sing{\operatorname{Sing}}
\def\tn{\t_{\rm null}}
\def\ni{\noindent}
\def\grad{\operatorname{grad}}
\def\dim{\operatorname{dim}}
\def\codim{\operatorname{codim}}

\def\Ts{\T_{\rm sing}}

\theoremstyle{plain}
\newtheorem{thm}{Theorem}

\newtheorem{prop}[thm]{Proposition}
\newtheorem{conj}{Conjecture}

\theoremstyle{definition}

\newtheorem{rem}[thm]{Remark}
\newtheorem{que}[thm]{Question}

\title[High multiplicity points of the theta divisor]{The loci of abelian varieties with points of high multiplicity on the theta divisor}
\author{Samuel Grushevsky}
\address{Mathematics Department, Princeton University, Fine Hall,
Washington Road, Princeton, NJ 08544, USA.}
\email{sam@math.princeton.edu}
\thanks{Research of the first author is supported in part by National Science Foundation under the grant DMS-05-55867.}
\author{Riccardo Salvati Manni}
\address{Dipartimento di Matematica, Universit\`a ``La Sapienza'',
Piazzale A. Moro 2, Roma, I 00185, Italy}
\email{salvati@mat.uniroma1.it}
\dedicatory{To the memory of Sevin Recillas}

\date{\today}
\begin{document}
\begin{abstract}
We study the loci of principally polarized abelian varieties with points of high multiplicity on the theta divisor. Using the heat equation and degeneration techniques, we relate these loci and their closures to each other, as well as to the singular set of the universal theta divisor. We obtain bounds on the dimensions of these loci and relations among their dimensions, and make further conjectures about their structure.
\end{abstract}

\maketitle

\section{Introduction}
Let ${\A_g}$ be the moduli space of complex $g$-dimensional
principally polarized abelian varieties (ppavs). Denote $\pi':\X_g\rightarrow {\A_g}$ the universal family of ppavs over it. Let $\T_g\subseteq \X_g$, and $\pi:\T_g\rightarrow{\A_g}$ denote
the universal theta divisor. We shall omit the index $g$ when it is clear.\smallskip

We can identify $\X$ and $\A$ respectively  with the spaces
$$
  \Sp\times\Z^{2g}\backslash\H_g\times\C^g,\quad \Sp\backslash\H_g,
$$
where $\H_g$ is the Siegel upper half-space. For $\tau\in{\H}$,  we denote by $X_{\tau}$, resp. $\T_\tau$, the fiber of $\pi'$, resp. $\pi$, over $\tau$ (more precisely, over the image of $\tau$ in $\A$.

A symmetric principal polarization $\T$ is the zero set of the holomorphic function (where $^tn$ stands for the transpose of a vector $n$)
$$
 \vartheta(\tau, z)=\sum_{n\in \Z^g}\exp(i\pi ({}^tn\tau n+2{}^tnz),
$$
and all the other principal polarization divisors symmetric under the $\pm 1$ involution of $X_\tau$ are obtained from this divisor by translating by points of order two $\frac{\tau\e+\de}{2}\in X_\tau[2]\subset X_\tau$, for some $\e,\de\in (\Z/2\Z)^g$.

We denote by $\Ts\subseteq\T$  the singular  locus on $\T$, i.e. the set of the classes of the points $(\tau, z)$ defined by
$$
 \Ts:=\lbrace (\tau,z)\in\H_g\times\C^g\mid \vartheta(\tau, z)=\p_i\vartheta(\tau, z)=\p_i\p_j\vartheta(\tau, z) =0\rbrace
$$
(where from here on we denote by $\p_i=\frac{\p}{\p z_i}$ the partial derivative in the $z_i$ direction). Moreover we  denote by $\S$  the union of the singular points of $\T_{\tau}$, i.e. the set of  the classes of the points $(\tau, z)$ defined by
$$
 \S:=\lbrace (\tau,z)\in\H_g\times\C^g\mid \vartheta(\tau, z)= \p_i\vartheta(\tau, z)=0\rbrace.
$$
Obviously we have
$$
  \Ts\subset\S\subset \T.
$$
There has been great interest in the singularities of the theta divisor and the loci of ppavs for which the theta divisor is singular at least since the ground-breaking work of Andreotti and Mayer \cite{am}, who defined what are now called the Andreotti-Mayer loci
$$
 N_k:=\pi_*(\S)=\{(X_{\tau}, \T_{\tau})\in{\A}\mid \dim\Sing\T_{\tau}\ge k\}.
$$
It is known that $N_0\subset\A$ is a divisor, which has at most two components, cf. \cite{mu,bea,deb}:
$$
 \t_{\rm null}:=\{ (X_{\tau}, \T_{\tau})\in{\A}\mid X_\tau[2]^{\rm even}\cap\Sing  \T_{\tau}\ne\emptyset\}
$$
$$
 N_0' :={\rm the\ closure\ of\ }\{ (X_{\tau}, \T_{\tau})\in{\A}\mid (X_\tau\setminus X_\tau[2]^{\rm even})\cap\Sing  \T_{\tau}\ne\emptyset\}
$$
(where $X_\tau[2]^{\rm even}$ denotes the even points of order two). The intersection of these two components was studied in \cite{deb,sing}.

In general the dimensions of the loci $N_k$ are not known. They were studied in detail by Ciliberto and van der Geer in \cite{cilvdg},\cite{amsp}, who conjecture that within the locus of ppavs with endomorphism ring $\Z$ (i.e. essentially with Picard group $\Z$) the codimension of any component of $N_k$ is equal to at least $\frac{(k+1)(k+2)}{2}$. They prove that for $g\ge 4$ and $1\le k\le g-3$ one has $\codim N_k\ge k+2$.

In this paper we will be interested in the loci of ppavs for which the theta divisor has points of higher multiplicity. For points of multiplicity two the natural loci to consider are
$$
 \begin{aligned}G_k:=
 &\{(X_{\tau}, \T_{\tau})\in{\A}\mid \dim\{z\in X_\tau : {\rm mult}_z\T_{\tau}\ge 2\}\ge k\}\\
=&\{(X_{\tau}, \T_{\tau})\in{\A}\mid
 \dim (X_\tau\cap\Ts)\ge k\}\end{aligned}
$$
(the two definitions are equivalent by the heat equation).

Already the locus $G_0$ is still a rather unknown object, and we will mostly concentrate on studying it. It has a natural subset
$$
 (\p\t)_{\rm null}:=\{ (X_{\tau}, \T_{\tau})\in{\A}\mid \X_{\tau}[2]^{\rm odd}\cap\Ts\ne\emptyset\}.
$$
We can further generalize this  to define
$$
 (\p^k\t)_{\rm null}:=\{ (X_{\tau}, \T_{\tau})\in{\A}\mid \exists x\in \X_{\tau}[2]: {\rm mult}_x\T_\tau- k\in 2\Z_{>0}\}
$$
For some computations it will be important to keep track of which dimension we are in. In this case we will use an upper index $(g)$ and write $\Ts^{(g)},(\p\t)_{\rm null}^{(g)},$ etc.

\smallskip
In this paper we give equations for  $(\p\t)_{\rm null}$ and $(\p^2\t)_{\rm null}$  using modular forms and so we can say that we solve Schottky's problem  for them. We describe explicitly some irreducible components of $(\p^k\t)_{\rm null}$, and thus obtain an estimate for their codimension in $\A_g$. Moreover we shall prove that these loci are reducible for $1\leq k\leq g-4$ (from \cite{el} it follows that $(\p^{g-2}\t)_{\rm null}$ is irreducible). Finally  we shall give some evidence for the expected dimensions of  $(\p\t)_{\rm null}$,  $G_0$ and $(\p^2\t)_{\rm null}$.  Doing this we  will relate the dimension of these three  varieties with the dimension of $\Ts$.

The methods we use consist mostly of working with theta functions and their derivatives, computing and bounding the dimensions of the tangent spaces by the ranks of explicit matrices of derivatives. We use the heat equation in many places, and compute the intersections of the loci we are interested in with the boundary of the partial compactification of $\A$. In \cite{amsp} Ciliberto and van der Geer study primarily the dimensions of the loci $N_k$, while we are mostly interested in the dimensions of $G_k$ and $(\p^k\t)_{\rm null}$. Perhaps uniting the two approaches may yield further insights into the geometry of the theta divisor.

\section{Notations}
We start by recalling  some basic facts about theta functions and modular forms. We denote $\H_g$ the {\it Siegel upper
half-space}, i.e. the set of symmetric complex $g\times g$ matrices
$\tau$ with positive definite imaginary part. Each such $\tau$
defines a complex   abelian variety   $\C^g/\tau\Z^g+\Z^g$. If $\sigma=\pmatrix a&b\\ c&d\endpmatrix\in\Sp$ is a symplectic matrix in a $g\times g$ block form, then its action on $\tau\in\H_g$ is defined by $\sigma\cdot\tau:=(a\tau+b)(c\tau+d)^{-1}$, and the moduli space of ppavs is the quotient $\A_g=\H_g/\Sp$. A period matrix $\tau$ is called {\it decomposable} if there exists $\sigma\in\Sp$ such that
$$
 \sigma\cdot\tau=\pmatrix \tau_1&0\\
 0&\tau_2\endpmatrix,\quad\tau_i\in\H_{g_i},\ g_1+g_2=g, g_i>0;
$$
otherwise we say that $\tau$ is indecomposable.

For $\e,\de\in (\Z/2\Z)^g$, thought of as vectors of zeros and ones,
$\tau\in\H_g$ and $z\in \C^g$, the {\it theta function with
characteristic $[\e,\de]$} is
$$
 \tt\e\de(\tau,z):=\sum\limits_{m\in\Z^g} \exp \pi i \left[
 ^t(m+\frac{\e}{2})\tau(m+\frac{\e}{2})+2\ ^t(m+\frac{\e}{2})(z+
 \frac{\de}{2})\right].
$$
We write  $\vartheta(\tau, z)$ for the theta function with
characteristic $[0,0]$. Observe that
$$
 \tt{0}{0}\left(\tau,z+\tau\frac\e2+\frac\de2\right)= \exp \pi i
 \left(-\frac{^t\e}{2}\,\tau\frac{\e}{2}\,\,-\frac{^t\e}{2}\,(z+
 \frac{\de}{2})\right)\tt\e\de(\tau,z),
$$
i.e. theta functions with characteristics are, up to some non-zero
factor, equal to $\vartheta(\tau,z)$ shifted by points of order two.

A {\it characteristic} $[\e,\de]$ is called {\it even} or {\it odd}
depending on whether  the scalar product
$\e\cdot\de\in\Z/2\Z$ is zero or one, respectively. The function
$\tt\e\de(\tau,z)$ is even or odd as a function
of $z$, according to the parity of the characteristic; thus a characteristic $[\e,\de]$ is even (resp. odd) if the multiplicity of the theta function $\vartheta(\tau,z)$ at the point $z=(\tau\e+\de)/2$ is even (resp. odd). We denote by $X_\tau[2]^{\rm even/odd}$ the set of even/odd points of order two on $X_\tau$.

A {\it theta constant} is the evaluation at $z=0$ of a theta function. All odd theta constants of course vanish identically in $\tau$, but their first derivatives at zero
do not vanish identically, and in fact transform in a nice way under the $\Sp$ action.

\smallskip
For a finite index subgroup $\Gamma\subset\Sp $ a multiplier system of weight $r/2$ is a map $v:\Gamma\to \C^*$, such that the map
$$
  \sigma\mapsto v(\sigma)\det(c\tau+d)^{r/2}
$$
satisfies the cocycle condition for every $\sigma\in\Gamma$ and
$\tau\in\H_g$.

\smallskip
Given a pair $\rho=(\rho_0,r)$, where $r$ is half integral, and $\rho_0:{\rm GL}(g,\C)\to \operatorname{End} V$ is an irreducible rational representation with the highest weight $(k_1,k_2,\dots,k_g)$, $k_1\geq k_2 \geq\dots\geq k_g=0$, we use the notation
$$
  \rho(A)=\rho_0(A)\det A^{r/2}\ .
$$

A map $f:\H_g\to V$ is called a modular form for $\rho$,
or simply a {\it vector-valued modular form}, if the choice
of $\rho$ is clear, {\it with multiplier $v$}, with respect to a
finite index subgroup $\Gamma\subset\Sp$ if
\begin{equation}\label{transform}
  f(\sigma\cdot\tau)=v(\sigma)\rho(c\tau+d)f(\tau)\qquad\forall
  \tau\in\H_g,\forall\sigma\in\Gamma,
\end{equation}
and if additionally $f$ is holomorphic at all cusps of
$\H_g/\Gamma$.

We define the {\it level}  and {\it Igusa}'s subgroups of the symplectic group to be
$$
  \Gamma_g(n):=\left\lbrace \sigma=\pmatrix a&b\\ c&d\endpmatrix
  \in\Sp\, |\, \sigma\equiv\pmatrix 1&0\\
  0&1\endpmatrix\ {\rm mod}\ n\right\rbrace
$$
$$
  \Gamma_g(n,2n):=\left\lbrace \sigma\in\Gamma_g(n)\, |\, {\rm
  diag}(a^tb)\equiv{\rm diag} (c^td)\equiv0\ {\rm mod}\
  2n\right\rbrace.
$$
When $n$ is even, these are finite index normal subgroups of $\Sp$.

\smallskip
\ni Under the action of $\sigma\in\Sp$ the theta functions transform
as follows:
$$
  \t\bmatrix \sigma\pmatrix \e\\ \de\endpmatrix\endbmatrix
  (\sigma\cdot\tau,\,^{t}(c\tau+d)^{-1}z)\qquad\qquad\qquad
$$
$$
  \qquad\qquad\qquad=\phi(\e,\,\de,\,\sigma,\,
  \tau,\,z)\det(c\tau+d)^{\frac{1}{2}}\tt\e\de(\tau,\,z),
$$
where
$$
  \sigma\pmatrix \e\cr \de\endpmatrix :=\pmatrix d&-c\cr
  -b&a\endpmatrix\pmatrix \e\cr \de\endpmatrix+ \pmatrix {\rm
  diag}(c \,^t d)\cr {\rm diag}(a\,^t b)\endpmatrix,
$$
considered in $(\Z/2\Z)^g$, and $\phi(\e,\,\de,\,\sigma,\,\tau,\,z)$ is some complicated explicit function. For more details, we refer to
\cite{igbook}. Thus theta constants with characteristics are (scalar) modular forms of weight $1/2$ with multiplier with respect to $\Gamma_g(2)$, i.e. we have
$$
  \tt\e\de(\sigma\cdot\tau,0)=v(\sigma, \e, \de) \det(c\tau+d)^{1/2}\tt\e\de(\tau,0)
  \qquad \forall \sigma\in\Gamma_g(2).
$$
where the multiplier $v$ becomes trivial if we assume  $\sigma\in\Gamma_g(4,8)$.

\smallskip
We call the {\it theta-null divisor} $\tn\subset\A_g$ the
zero locus of the product of all even theta constants. We  denote  by $\grad\tt\e\de(\tau)$ the gradient of the theta function of characteristic $[\e,\,\de]$  with respect to $ z_1,\dots,z_g$ and evaluating at $z=0$. This gradient is not identically zero if and only if the characteristic is odd. The gradient is a vector valued modular form  for $\rho=(St, 1/2)$  with multiplier $v$, cf.\cite{gs}.

We recall that theta functions  (and their derivatives) satisfy the heat equation, i.e.
$$
 \frac{\p^2\tt\e\de(\tau,z)}{\p z_j\p z_k} =2\pi
 i(1+\delta_{j,k})\frac{\p\tt\e\de(\tau,z)}{\p\tau_{jk}},
$$
(where $\delta_{j,k}$ is Kronecker's symbol).

\smallskip
The   symmetric matrix  associated to the second derivatives
\begin{equation}\label{2z}
 \left(\frac{\p^2\tt\e\de(\tau,z)}{\p z_j\p z_k}|_{z=0}\right)
\end{equation}
is a  vector valued modular form for $\Gamma_g(2)$ if we restrict  to the locus $\tt\e\de(\tau,0)=0$.
This is a general fact: the derivative of a section of a line bundle is a section of the same bundle when restricted to the zero locus of the section --- the modularity of this particular gradient is discussed in \cite{genus4}. Similarly note that third derivatives of a theta function with an odd characteristics for a modular form,
when restricted to the  locus   $\grad\tt\e\de(\tau,0)=0$.

\section{Equations for the loci $(\p^k\t)_{\rm null}$}
Similarly to the case of $\tn$, we give vector valued (and alternatively, scalar valued) equations for the loci
$(\p\t)_{\rm null}$ and  $(\p^2\t)_{\rm null}$. We will work on the level covers $\A_g(2):=\H_g/\Gamma_g(2)$ or
$\A_g(4, 8):=\H_g/\Gamma_g(4, 8)$, where the $\tn$ divisor decomposes into a union of components corresponding to the individual characteristics. Note also that  $\A_g(4,8)$ is a smooth manifold cover of the stack/orbifold $\A_g$, so working on the level cover takes care of the stackiness.

On $\A_g(2)$ for any odd $[\e, \d]$ the vector valued equation
$\grad\tt \e\d(\tau) =0$ defines a certain set of components of $(\p\t)_{\rm null}.$ The (possibly reducible) loci $\grad\tt \e\d(\tau) =0$ for various $\e,\d$ are conjugate under the action of $\Sp$.

For any $[\e_1,\,\d_1],\dots,[\e_g,\,\d_g]$ we define
$$
 D([\e_1,\,\d_1],\dots,[\e_g,\,\d_g])(\tau):=\grad \tt{\e_1}{\d_1}\wedge\grad\tt{\e_2}{\d_2}\wedge \dots\wedge \grad\tt{\e_g}{\de_g}(\tau)
$$
which is a scalar modular form with multiplier of weight $\frac{g+2}{2}$ with respect to $\Gamma_g(2)$, cf. \cite{sm}.

It follows from Lefschetz theorem  for abelian varieties, cf. \cite{gs}, that the $g\times 2^{g-1}(2^g -1)$ matrix
$$
 \left(\dots\grad\tt \e\d(\tau)\dots\right)_{{\rm all\, odd} [\e,\,\de]}
$$
has maximal rank for all $\tau$. Thus if all its minors including a fixed characteristic $[\e,\d]$ vanish, the corresponding gradient must be zero, and we get
\begin{prop}
The common zero locus of the scalar modular forms
$$
 D([\e,\,\de],[\e_2,\,\de_2], \dots([\e_g,\,\de_g](\tau)
$$
for $[\e,\de]$ fixed, and all possible odd characteristics $[\e_2,\,\d_2],\dots, [\e_g,\,\d_g]$,
is equal to the locus $\grad\tt \e\d(\tau) =0$.
\end{prop}
\begin{prop}
Similarly, if $[\e, \d]$ is an even characteristic,
the common zero locus of the scalar valued equation
\begin{equation}\label{theta}
  \tt \e\d(\tau)=0
\end{equation}
and the vector valued equations
\begin{equation}\label{theta2}
 \left((1+\d_{i, j})\frac{\p^2\tt\e\de}{\p\tau_{i\,j}}\tt\a\b- (1+\d_{i,j})\frac{\p^2\tt\a\b}{\p\tau_{i\,j}}\tt\e\de\right)(\tau)=0
\end{equation}
for all even characteristics  $[\a,\, \b]$, defines a union of some irreducible components of  $(\p^2\t)_{\rm null}$ on $\A_g(2)$. (To see this, one notes that all even theta constants cannot vanish simultaneously, cf. \cite{genus4}.)
\end{prop}

Note that of course the components of $(\p^2\t)_{\rm null}$ on $\A_g(2)$ are given by
\begin{equation}\label{theta2bis}
 \left((1+\d_{i,j})\frac{\p^2\tt\e\de}{\p\tau_{i\,j}}\right)(\tau)=0;
\end{equation}
however, this expression is a modular form only along the divisor $ \tt \e\d(\tau,0)=0$, cf. \cite{genus4}.

There is also another method to produce scalar valued modular forms  vanishing exactly on the  component of  $(\p^2\t)_{\rm null}$. This method is  similar to the method used for  $(\p\t)_{\rm null}$. Indeed, the $2^{g-1}(2^g+1)\times\left(\frac{g(g+1)}{2}+1\right)$ matrix
$$
\left(\begin{array}{ccc}
  \dots &\tt \e\d&\dots \\
\dots&\dots&\dots\\
 \dots& {\p\tt\e\de}/{\p\tau_{i\, j}}&\dots
  \end{array}\right)_{{\rm all\, even} [\e,\,\de]}
$$
has maximal rank, cf.\cite{igbook}. Hence, setting $N=(1/2)g(g+1)$, the form
$$
 D^2([\e,\,\de],[\e_1,\,\de_1]\dots[\e_N,\,\de_N])(\tau):=
$$
$$
 {\rm det}
 \left(\begin{array}{llll}
  \ \ \tt \e\d&\ \ \tt{\e_1}{\de_1}&\dots&\ \ \tt{\e_N}{\de_N}\\
 {\p\tt\e\de}/{\p\tau_{1\, 1}}&{\p\tt{\e_1}{\d_1}}/{{\p \tau_{1\, 1}}}&\dots  &{\p\tt{\e_N}{\d_N}}/{{\p \tau_{1\, 1}}}\\
\dots&\dots&\dots&\dots\\
 {\p\tt\e\de}/{\p\tau_{i\, j}}&{\p\tt{\e_1}{\d_1}}/{{\p \tau_{i\, j}}}&\dots  &{\p\tt{\e_N}{\d_N}}/{{\p \tau_{i\, j}}}\\
\dots&\dots&\dots&\dots\\
 {\p\tt\e\de}/{\p\tau_{g\, g}}&{\p\tt{\e_1}{\d_1}}/{{\p \tau_{g\, g}}}&\dots  &{\p\tt{\e_N}{\d_N}}/{{\p \tau_{g\, g}}}\\
  \end{array}\right)(\tau) ,
$$
is a modular form with multiplier, of weight $g+1+(1/2)(N+1)$ relatively to $\Gamma_g(2)$, cf. \cite{diff}, and similarly to the previous case we have
\begin{prop}
The modular forms
$$
 D^2([\e,\,\de],[\e_1,\,\de_1]\dots[\e_N,\,\de_N])(\tau)
$$
for $[\e,\de]$ and all possible even characteristics $[\e_1,\,\d_1],\dots, [\e_N,\,\d_N]$, vanish  simultaneously along the locus defined by (\ref{theta}) and (\ref{theta2bis}), and give scalar equations for one  component of the locus $(\p^2\t)_{\rm null}$.
\end{prop}
\begin{rem}
One could also write more complicated equations for the locus $(\p^k\t)_{\rm null}$ for $k\geq 3$, using suitable vector-valued modular forms. Unfortunately our method for obtaining equations using scalar modular forms does not work in this case, since the jacobian matrices do not have maximal rank everywhere, cf. \cite{diff}.
\end{rem}

\section{Some components of $(\p^k\t)_{\rm null}$ within the locus of decomposable ppavs}
In this section we start our investigation of the irreducible components of  $(\p^k\t)_{\rm null}$ and their possible dimensions.

As an immediate consequence of the results of \cite{ko} or \cite{el}, we have
\begin{prop}
$$(\p^{k}\t)_{\rm null}=\emptyset \quad {\rm for}\, k\geq g-1$$
$$(\p^{g-2}\t)_{\rm null}=\A_1\times\dots\times\A_1 $$
\end{prop}
For lower values of $k$ we can describe some components:
\begin{thm}\label{irred}
For $1\leq k\leq g-2$, the variety
$\t_{g-k,\,\rm null}  \times\A_1\times\dots \times A_1$  is an  irreducible components of $(\p^k\t)_{\rm null}$.
The codimension of this subvariety gives the bound
$$
 \codim(\p^k\t)_{\rm null} \leq gk+1-(1/2)(k^2+k)
$$
\end{thm}
\begin{proof}
We perform the computation on $\A_g(4,8)$ for the characteristic
$\e=(\a,1, \dots,1)$ and $\de=(\b,1,\dots,1)$, where $[\a,\b]$ is a $(g-k)$-dimensional characteristic. Note that $[\a,\b]$ is necessarily even. Consider then the set
$$
 A_k([\e,\de]):=\left\{\frac{\partial^h\tt\e\de(\tau, z)}{\partial z_{i_1}\dots\partial z_{i_h}}\vert_{z=0}=0\right\}_{{\rm for\ all}
 \ h\leq k.}
$$
Obviously $A_k$ contains the $\Gamma_g(2)$ conjugates of the locus
$$
 \left(\tt\a\b(\tau')=0\right)\times \H_1\times\dots\times\H_1,
$$
(where $\tau'\in\A_{g-k}$),
which is of codimension $ gk+1-(1/2)(k^2+k)$ in $\A_g$. The only non-zero elements of the jacobian matrix of the equations defining $A_k$ are those involving derivatives of  order $k+2$. These form a matrix with $(1/2)(g+1)g$ rows and the columns that can be indexed in the following three ways:\smallskip

1) the $k$ derivatives involve the indices $g-k+1,g-k+2,\dots, g$

2) the $k$ derivatives involve only one of the first $g-k$ indices and all, but one among $g-k+1,g-k+2,\dots, g$

3) the $k$ derivatives involve two ( even with multiplicity) of the first $g-k$ indices and all, but two among $g-k+1,g-k+2,\dots, g.$

Then we have to consider the derivative $\p/\p \tau_{a,b}$,  or equivalently, by the heat equation, $\p^2/\p z_a \p z_b$. In  the first of the above cases we get a column with non-zero entries being
$$
 \frac{\partial^2  \tt\a\b(\tau')}{\partial z_{a}\partial z_{b} }\prod_{i=1}^{k}  \left(  \tt 1 1^{'}(\lambda_i)\right)\quad {\rm with}\,\,1\leq a\leq b\leq g-k.
$$
In the second case we get $(g-k)k$ columns involving in the $(a,b)$ row
$$
  \frac{\partial^2  \tt\a\b(\tau')}{\partial z_{a}\partial z_{j} }\prod_{i=1}^{k}  \left(  \tt 1 1^{'}(\lambda_i)\right)\quad {\rm with}\,\,1\leq a \leq g-k<b\leq g.
$$
In the third case we get $1/2)k(k-1)$ columns involving in the $(a,b)$ row
$$
 \frac{\partial^2  \tt\a\b(\tau')}{\partial z_{l}\partial z_{j} }\prod_{i=1}^{k}  \left(  \tt 1 1^{'}(\lambda_i)\right)\quad {\rm with}\,\, g-k<a< b\leq g,
$$
where we denoted $\tt 1 1^{'}(\lambda_i)$ the derivative of the one-dimensional theta function with odd characteristic evaluated at zero for the elliptic curve with periods 1 and $\lambda_i=\tau_{g-k+i,g-k+i}$.

It can be easily  checked that all these columns are linearly independent, and hence the Jacobian has maximal rank, so that the equations give an irreducible component of the locus $(\p^k\t)_{\rm null}$, as claimed.
\end{proof}

\begin{prop}
When $k\leq g-4$, the locus $(\p^k\t)_{\rm null}$ is reducible.
\end{prop}
\begin{proof}
We recall that  $\t_{2,\,\rm null} =\A_1\times \A_1$ and $\t_{3,\,\rm null}$ is the hyperelliptic locus in $\A_3$ (in particular a generic element of it is indecomposable). Thus the locus
$$
 \t_{g-k-1,\,\rm null}  \times\A_2\times\A_1\times\dots\times\A_1 \subset (\p^k\t)_{\rm null},
$$
and must be contained in an irreducible component of $(\p^k\t)_{\rm null}$ different from $\t_{g-k,\,\rm null}  \times\A_1\times\dots \times A_1$.
\end{proof}

\begin{rem}
When $k=g-3$,  for small $g$, we have that $(\p^{g-3}\t)_{\rm null}$  is irreducible. In general this is not known. Repeating  the argument of the proof of corollary 2 in   \cite{el}, one can show that if a principally polarized abelian variety in  $(\p^{g-3}\t)_{\rm null}$  is decomposable, and thus necessarily lies in  $\t_{g-3,\,\rm null}  \times\A_{1}\times\dots\times\A_{1}$. It is interesting to know whether there are indecomposable ppavs in $(\p^{g-3}\t)_{\rm null}$. See also remark \ref{gminus3}
\end{rem}

\begin{rem}
For a generic $\tau\in\t_{g-k,\,\rm null}  \times\A_1\times\dots \times A_1$ the theta divisor $\T_{\tau}$ has a unique point of order 2 having  multiplicity $k+2$. However, for a generic point in
$\t_{g-k-n,\,\rm null}  \times\A_{i_1}\times\dots\times\A_{i_k}$, with $k<i_1+\dots i_k=k+n\leq g-2$, the theta divisor $\T_{\tau}$ has several points of order 2 and multiplicity $k+2$.
As a consequence we see that the latter product is contained in the    singular locus of  $(\p^k\t)_{\rm null}\subset\A_g(2)$.
\end{rem}

We now investigate the reducibility of the locus $G_0$.
\begin{thm}\label{reduc}
The loci $N_{g-1,\,0} ' \times\A_1$ and $\t_{g-1,\,\rm null}  \times\A_1$ are irreducible components of $G_0$.
\end{thm}
\begin{proof}
Recall that the locus $\Ts\subset\X_g$ is defined by the following equations (using the heat equation for the theta function)
$$
 0=\vartheta(\tau,z)=\p_i\vartheta(\tau,z)=\p_i\p_j\vartheta(\tau,z)\quad\forall 1\le i,j\le g.
$$
Computing the Jacobian matrix of these defining equations for $\tau\in\A_{g-1}\times\A_1$ we get
$$
 \left(\begin{array}{cccccccc}
 \,\p_1\vartheta&\p_1\p_1\vartheta
 &\dots&\p_1\p_g\vartheta&\p_1\p_1\p_1\vartheta&\dots&\p_1\p_g\p_g \vartheta\\
 \,\p_2\vartheta&\p_2\p_1\vartheta
 &\dots&\p_2\p_g\vartheta&\p_2\p_1\p_1\vartheta&\dots&\p_2\p_g\p_g \vartheta\\  \dots&\dots&\dots&\dots&\dots&\dots&\dots \\\
 \,\p_g\vartheta&\p_g\p_1\vartheta&\dots&\p_g\p_g\vartheta& \p_g\p_1\p_1\vartheta&\dots&\p_g\p_g\p_g \vartheta\\
 \,\p_1\p_1\vartheta&\p_1\p_1\p_1\vartheta
 &\dots&\p_1\p_1\p_g\vartheta&\p_1\p_1\p_1\p_1\vartheta&\dots& \p_1\p_1\p_g\p_g\vartheta\\
   \dots&\dots&\dots&\dots&\dots&\dots \\\
  \,\p_g\p_g\vartheta&\p_g\p_g\p_1\vartheta
 &\dots&\p_g\p_g\p_g\vartheta&\p_g\p_g\p_1\p_1\vartheta&\dots& \p_g\p_g\p_g\p_g\vartheta
 \end{array}\right)
$$
In both cases, let $$(X_{\tau}, \T_{\tau})=(X_{\tau'}, \T_{\tau'})\times(E_{\lambda}, \T_{\lambda})$$
and let $x=(x', y)$ be the point of multiplicity at least 3 such that
$x'$ is a singular point of  $\T_{\tau'}$ and $y= \T_{\lambda}$. We can assume that $x'$ is an ordinary double point, then we have that the jacobian  matrix at the point $(\tau, x)\in \Ts$ is of the form
$$
  \left(\begin{array}{ccc }
  0&0&\p_i\p_j\p_k\vartheta\\
  0&{}^t(\p_i\p_j\p_k\vartheta)&\p_i\p_j\p_k\p_l\vartheta
  \end{array}\right)
$$
Here the first column is in fact a column-vector, the entry (1,2) is the $g\times g$ null matrix; $(\p_i(\p_j\p_k\vartheta))$ is  the  $g\times (1/2)g(g+1)$ matrix of the third derivatives of $\vartheta$, ${}^t(\p_i\p_j\p_k\vartheta)$ is its transpose, and  the entry  (2,3)  is the $(1/2)g(g+1)\times (1/2)g(g+1)$ matrix of the fourth derivatives of $\vartheta$. An elementary computation gives that the rank of the matrix $(\p_i\p_j\p_k\vartheta)$ is $g$, since $x'$ is an ordinary double point, hence the previous matrix has rank $2g$.

Hence  in a neighborhood of the point  $(\tau, x)$ the subvariety $\Ts$ has codimension at least $2g$ in $\X_g$, and thus its projection has codimension at least $g$ in $\A_g$. Since ${N}_{g-1,\,0} ' \times\A_1$ and $\t_{g-1,\,\rm null}  \times\A_1$ both have codimension $g$ in $\A_g$, these are two irreducible components of $G_0$.
\end{proof}

In more generality, note that $(\p^k\t)_{\rm null}\times\A_1\subset (\p^{k+1}\t)_{\rm null}$, and we will now consider the question of whether these give irreducible components. The result is the following generalization of theorem \ref{irred}.

\begin{prop}\label{codg}
Let $X$ be an irreducible component of $(\p^k\t)_{\rm null}^{(g)}$.
Then the locus $X\times \A_1$  is an irreducible component of $ (\p^{k+1}\t)_{\rm null}^{(g+1)}$ if and only if the rank of the
$g \times \binom{g+k}{k+1}$ matrix
$$
 L:=\left(\frac{\p\tt\e\de (Z)}{\p z_{a}\p z_I}\right)
$$
for $1\leq a \leq g$ and $I\subset\{1,\dots,g\}$ with $\#I=k+1$, is equal to $g$.

Moreover, if $k=1$, the set $X\times \A_1$  is an irreducible component of $ (\p^{2}\t)_{\rm null}$ if and only if the codimension of $X$ in $\A_g$ is equal to $g$.
\end{prop}
\begin{proof}
As before, we will compute the Jacobian matrix of the defining equations at a smooth point $\tau\in X$, lifted to the level cover $\A_g(4,8)$, so that we are talking about the multiplicity of the theta function with a given characteristic $[\a,\b]$ at zero. The   $(1/2)g(g+1)\times \binom{g+k-1}{k}$ Jacobian matrix is then
$$
  J=\left(\frac{\partial^2  \tt\a\b(\tau)}{\partial z_{a}\partial z_{b} \p z_I}\right)
$$
with $1\leq a\leq b\leq g$, and $I\subset\{1,\dots,g\},$ with $\#I=k$. Since this is the Jacobian matrix of the equations defining $X$, its rank is equal to the codimension of $X$ in $\A_g$.

Then for any $x\in\A_1$ the point $(\tau,x)\in \A_g\times \A_1$ is a smooth point of $(\p^{k+1}\t)_{\rm null}$, with the vanishing theta function with high multiplicity being $\tt{\a\ 1}{\b\ 1}$. The Jacobian matrix of the equations defining $\p^{k+1}\t)_{\rm null}\subset\A_{g+1}$ is the same as $J$ together with the matrix $L$. The first statement of the proposition follows, while for the special case of $k=1$ we observe that $L$ is the transpose of $J$ and thus has the same rank.
\end{proof}

\section{Codimensions of $(\p\t)_{\rm null},(\p^2\t)_{\rm null}$, and $G_0$}
Note that the dimension of a reducible variety means the maximum dimension of the components, and thus $\codim$ means the minimal codimension of the components.

In the previous section we have computed some components of the loci $(\p^k\t)_{\rm null}$ and thus know their dimensions. It is natural to conjecture the following.
\begin{conj}\label{dthetanull}
For any $g\geq 3$
$$
 \codim_{\A_g} (\p\t)_{\rm null}=\codim_{\A_g} G_0=g.
$$
\end{conj}

We give some evidence for this conjecture by using the degeneration techniques. We want to consider the  intersection of $ (\p\t)_{\rm null}$ with $\p\A_{g+1}$, the corank one boundary component of any toroidal compactification of $\A_{g+1}$. A point of $\p\A_{g+1}$ is a semiabelian variety of torus rank one, i.e. a pair $(\overline{G}, D)$, where $\overline{G}$ is a non-normal compactification of the extension $1\to{\mathbb G}_m\to G\to B\to 0$ for some $B\in\A_g$, and $D$ is an ample divisor that is the limit of the theta divisors. The restriction of the map
\begin{equation} \label{fib}
 \pi_{|\p\A_{g+1}}:\p\A_{g+1}\to\A_{g}
\end{equation}
has $B/Aut (B, \Xi)$ as fiber over $(B,\Xi)$. This means that the
fiber over a general $(B,\Xi)\in\A_{g}$ is the Kummer variety
$B/\pm 1$.

Studying the boundary behavior of the universal theta divisor and using the Fourier-Jacobi expansion similarly to \cite{mu}, we get the following
\begin{prop}
\begin{equation}\label {inter0}
 (\p\t)_{\rm null}^{(g+1)}\cap\partial\A_{g+1}=\left(\bigcup_{ (\overline{G},D)}2_B(\S^{(g)}\cap\pi^{-1}((B,\Xi))\right)\cup \left(\pi^{-1} (\p\t)_{\rm null}^{(g)}\cap\T_g\right)
 \end{equation}
\end{prop}
\begin{proof}
Analytically   we consider the Fourier-Jacobi expansion of
$$
 \grad\tt\e\d(Z,0)\quad{\rm with}\quad Z=\left(\begin{array}{cc}
\tau&z\\
z'&w
 \end{array}\right)
$$
along the boundary components. Since the symplectic group acts transitively on the boundary components and on the odd characteristics, considering the  Fourier-Jacobi expansion of  a fixed $\grad\tt\e\de$ at all boundary components is equivalent to considering the Fourier-Jacobi expansion of $\grad\tt\e\de$ for all possible $[\e,\de]$ at the standard boundary component where the period matrix element $\tau_{g+1,g+1}$ goes to $i\infty$. It is easy to check that the Fourier-Jacobi expansions are of two types depending on whether the $(g+1)$'th entry  of the characteristic is 0 and 1. If it is 1  one we get
$$
 \p_i\,  \tt \e\d(Z, 0) =\p_i\tt {\a}{\b}(\tau, z/2)exp(i\pi w/4)+\dots\quad {\rm when}\,i=1,\dots, g
$$
$$
 \p_{g+1}\,  \tt \e\d(Z, 0) =\tt {\a}{\b}(\tau, z/2)exp(i\pi w/4)+\dots
$$
If it is 0 we get
$$
  \p_i\,  \tt \e\d(Z, 0) =\p_i\tt {\a}{\b}(\tau, 0)+\dots
$$
$$
  \p_{g+1}\,  \tt \e\d(Z, 0) =\tt {\a}{\b}(\tau, z)exp(i\pi w)+\dots
$$
Now the intersection  of $\grad\tt\e\de=0$ with the blow up of the first boundary components gives $2_{X_{\tau}}(\Sing (\T_{\tau}))$ along each fiber, and $\T_{\tau}$ along $p^{-1}((\p\t)_{\rm null}^{(g)})$
\end{proof}

Notice that by results of Keel and Sadun \cite{ke}, any component of $(\p\t)_{\rm null}^{(g)}$, which is a priori of codimension at most $g$, must intersect the boundary of any toroidal compactification of $\A_g$. We can handle the degenerations only for the case when the components intersect the torus rank one boundary component, i.e. when they intersect the boundary of the partial compactification.
\begin{thm}\label{boun}
If all irreducible components of $(\p\t)_{\rm null}^{(g)}$  meet $\p\A_g$ (the boundary of the partial compactification),  and $\codim_{\A_{g-1}} (\p\t)^{(g-1)}_{\rm null}=g-1$, then   $\codim_{\A_g} (\p\t)_{\rm null}=g$, i.e. if for any $g$ all irreducible components of $(\p\t)_{\rm null}^{(g)}$ meet $\p\A_g$, then conjecture \ref{dthetanull} holds.
\end{thm}
\begin{proof}
Obviously we  know that the codimension is less or equal to $g$.
The first case  in the intersection above gives a subvariety of codimension $g$. By induction the second case also gives a subvariety of codimension $g$, since $(\p\t)_{\rm null}^{(g-1)}$ has codimension $g-1$ (and thus also its preimage under $\pi^{-1}$), and we are then intersecting with a divisor.
\end{proof}

One can also conjecture the codimensions of the other loci we study.
\begin{conj} \label{co1}
For any $g\geq 3$
$$
 \codim_{\X_g}\Ts=2g.
$$
\end{conj}

\begin{conj} \label{co2}
For any $g\geq 4$
$$
  \codim_{\A_g}(\p^{2}\t)_{\rm null}= 2g-2.
$$
\end{conj}

We will give a degeneration argument relating these two conjectures.
\begin{thm}\label{muthm}
\begin{equation}\label {inter}
 (\p^{2}\t)_{\rm null}^{(g+1)}\cap\partial\A_{g+1}=\left(\bigcup_{ (\overline{G},D)}2_B(\Ts\vert_{\Xi})\right)\cup (\pi^{-1}(\p^2\t)_{\rm null}^{(g)}\cap \S^{(g)})
 \end{equation}
\end{thm}
\begin{proof}
The proof is analogous to the proof of the identity (\ref{inter0}), if one considers the second derivatives as well. There are no new technicalities or ideas involved, and we omit the details here.
\end{proof}

\begin{rem}
This intersection give us a rough estimate for the dimension of $ (\p^{2}\t)_{\rm null}^{(g+1)}$. In fact the intersection  of any irreducible components of  $ (\p^{2}\t)_{\rm null}^{(g+1)}$,  meeting  the first  boundary component $\p\A_{g+1}$ have dimension less or equal to $\dim\,\S^{(g)}=(1/2)g(g+1)-1$. This implies that all
irreducible components of  $ (\p^{2}\t)_{\rm null}^{(g+1)}$  meeting  the first  boundary component have codimension at least
$g+1$.
\end{rem}

We can relate the two conjectures by the following theorem
\begin{thm}
If for any $k\in[4,g+1]$ all components of $(\p^2\t)_{\rm null}^k$ intersect $\p\A_{k}$, then conjecture $\ref{co2}$ is true for all $k\in[4,g+1]$ if and only if conjecture $\ref{co1}$ is true for all $k\in[3,g]$.
\end{thm}
\begin{proof}
We use the description of the boundary of the locus $(\p^2\t)_{\rm null}$ from formula (\ref{inter}). We know a priori that $\Ts^{(g)}$ has codimension at most $2g$ (since its projection to $\A_g$  contains $(\p\t)_{\rm null}$, which is of codimension at most $g$). Similarly we know a priori that $\codim(\p^2\t)_{\rm null}\le 2g-2$, since this locus contains $\tn^{(g-2)}\times\A_1\times\A_1$.

One implication is easy: if conjecture \ref{co2} holds for $k=g+1$, i.e. $\codim_{\A_{g+1}}(\p^2\t)_{\rm null}^{(g+1)}=2g$, then $\codim_{\p\A_{g+1}}\left((\p^2\t)_{\rm null}^{(g+1)}\cap \p\A_{g+1}\right)=2g$ (we know from (\ref{inter}) that this intersection with an open divisor $\p\A_{g+1}\subset\bar{\A_{g+1}}$ is non-empty, so it has codimension one in $(\p^2\t)_{\rm null}^{(g+1)}$), and thus the codimension of its component $2_B(\Ts^{(g)})$ cannot be less than $2g$.

For the other direction, assume that conjecture \ref{co1} holds for $k=g$, i.e. that $\codim_{\X_g}(\Ts)=2g$, and also inductively assume that we know that $\codim (\p^2\t)_{\rm null}^{(g)}=2g-2$. Since for any $\tau\in\A_g$ we have $\codim_{X_\tau} \S\cap X_\tau\ge 2$ (a generic point of $\Theta_{\tau}$ is smooth). Thus we have $\codim_{\X_g} \pi^{-1}(\p^2\t_{\rm null}^{(g)}\cap\S^{(g)})\ge 2g$, and thus $\codim_{\A_{g+1}}(\p^2\t)_{\rm null}^{(g+1)}=2g$.

For the base of the induction, observe that the locus $\Ts^{(3)}$ consists of the family of odd points of order two over $\A_1\times\A_1\times\A_1 $ and so has codimension $6$ in $\X_3$, while $(\p^2\t)_{\rm null}^{(4)}=\A_1\times\A_1\times\A_1\times\A_1=\tn^{(2)}\times\A_1\times\A_1$ also has codimension $6$ (see theorem \ref{irred}).
\end{proof}

\begin{rem}
In lemma 9 of \cite{sing} we proved and used  that $ (\p^{2}\t)_{\rm null}$ has codimension greater than $2$ in $\A_{g+1}$. Our proof quoted a result of  \cite{amsp} the proof of theorem 8.11 in \cite{amsp}, while it seems to us the statement in theorem 8.6(ii) of \cite{amsp} yields this more directly, implying that $3\le\codim G_0\le\codim (\p^{2}\t)_{\rm null}$.
\end{rem}

\medskip
In analogy with the situation for Andreotti-Mayer loci, it is natural to make the following
\begin{conj} \label{co3}
For any $g\geq 3$
$$
 \codim G_k \ge g+k
$$
\end{conj}
\begin{rem}
Notice that we know that the locus $(\p\t)_{\rm null}$ is by definition a subset of $G_0$; however, $(\p\t)_{\rm null}$ is locally given by the vanishing of a gradient of an odd theta function with characteristic at zero, i.e. is given by $g$ equations, and thus we have $g\le\codim (\p\t)_{\rm null}\le\codim G_0$, so the above conjecture implies that $\codim G_0=\codim (\p\t)_{\rm null}=g$.
\end{rem}

\begin{rem}
In fact it seems likely that the locus $(\p\t)_{\rm null}$ is of pure codimension $g$, while it seems less likely that $G_0$ is of pure dimension: the second derivatives of theta automatically vanish at odd points of order two, but not generically on $\X_g$, and thus it is natural to expect that there exist components of $G_0\setminus(\p\t)_{\rm null}$ are of codimension higher than $g$.
\end{rem}
\begin{rem}\label{gminus3}
In the notations of \cite{amsp} the loci $N_{g,k,3}=G_k^{(g)}$ are studied, and the following results are know for them. Theorem 8.6(ii) in \cite{amsp} then says that $\codim G_k\ge k+3$, while theorem 7.5(ii) in \cite{amsp} in our notations is that $$G_{g-3}^{(g)}=\mathop{\bigcup}\limits_{g_1+g_2+g_3=g,
\ g_1g_2g_3>0}\A_{g_1}\times\A_{g_2}\times\A_{g_3},$$
which is not purely dimensional, but has codimension $2g-3$ --- so that conjecture \ref{co3} is true for $k=g-3$.
\end{rem}

The statement  about the codimension of $G_k$ seems to be   rather convincing, but we could not find a proof.  Note that the argument to prove $\codim N_1>1$ given by Mumford in \cite{mu} uses an involved heat equations argument, while Ciliberto and van der Geer in \cite{cilvdg}, \cite{amsp} show, with a lot of work, that $\codim N_k\ge k+2$. The basic question seems to be whether the locus $\Ts$ is pure-dimensional or not, and the relationship of the conjecture is the following
\begin{prop}
Conjecture \ref{co1}, for a fixed $g$, is equivalent to conjecture \ref{co3}, for the same $g$, and all $k$.
\end{prop}
\begin{proof}
Indeed recall that by definition the map $\pi$ restricted to $\Ts\cap\pi^{-1}(G_k)$ has fiber dimension $k$, and thus $\dim\Ts\ge k+\dim G_k$. Thus if we know that $\codim_{\X_g}\Ts=2g$, it follows that $\codim_{\A_g}G_k\ge g+k$. In the other direction, if $\codim G_k\ge g+k$ for all $k$, we have $\dim\Ts=\max_k(k+\dim G_k)\le \max_k(k+\dim{\A_g}-g-k)$, and thus $\codim\Ts\ge 2g$, in which case it must be equal to $2g$.
\end{proof}

Now we would like to give some evidence for the validity of these conjectures. Note that if $z=(\tau\e+\de)/2\in X_\tau[2]$ is an odd point, then $0=\vartheta(\tau,z)=\p_i\p_j\vartheta(\tau,z)$ automatically as the value and derivatives of an odd function, and thus the locus
$$
 Y:=\{(\tau,z)\in\X_g \mid \grad \vartheta(\tau,z)=0;\ z=(\tau\e+\de)/2\}
$$
is a subset of $\Ts\subset\X_g$, while the projection $\pi(Y)=(\p\t)_{\rm null}\subset\A_g$ (notice that by definition $(\p\t)_{\rm null}$ is the union of such projections over all odd $[\e,\de]$, but they are all conjugate under $\Sp$, and thus have the same image on $\A_g$). The expected codimension of $Y$ in $\X_g$ is equal to $2g$ ($g$ for fixing a point on $X_\tau$, and $g$ for the vanishing of the gradient).
\begin{thm}\label{tch}
Let $Z$ be a reduced irreducible component of $\Ts$ that is contained in $Y$ as above. Then $Z$ has codimension $2g$ in $\X_g$ and $\pi(Z)$ has codimension $g$ in $\A_g$ (and is thus an irreducible component of $(\p\t)_{\rm null}$).
\end{thm}
\begin{proof}
We apply  the jacobian criterion, in a smooth point  of $Z$. As in the theorem \ref{reduc}, the dimension of the tangent space to $\Ts$ is $(1/2)g(g+1)-g+2k$ if $rk (\p_i\p_j\p_k\vartheta)=g-k .$

On the other hand, the normal space to $Y$ at the same point is given by the matrix
$$
  \left(\begin{array}{cc }
  0&1_g\\
  (\p_i\p_j\p_k\vartheta)'&M(\e)
  \end{array}\right)
$$
with $M(\e)$  a $(1/2)g(g+1)\times g$   matrix  depending on $\e$. This matrix has rank $2g-k$, hence the dimension of the tangent space is $(1/2)g(g+1)-g+k$. Since the  two computations    give the same result , we get $k=0$.  We get that the codimension of $\pi(Y)$ is at least $g$. Since on a suitable covering of $\A$ we have that this locus is defined by the equations
$$
 \p_i \tt \e\d(\tau, 0)=0,
$$
We have that the codimension is at most $g$, hence it is exactly $g$.
\end{proof}

\end{document}